\documentclass[12pt]{amsart}
\usepackage{amscd,amsmath,amsthm,amssymb}
\usepackage{pstcol,pst-plot,pst-3d}
\usepackage{color}
\usepackage{pstricks}
\usepackage{stmaryrd}
\usepackage{tikz}
\usepackage{tikz-cd}

\newpsstyle{fatline}{linewidth=1.5pt}
\newpsstyle{fyp}{fillstyle=solid,fillcolor=verylight}
\definecolor{verylight}{gray}{0.97}
\definecolor{light}{gray}{0.9}
\definecolor{medium}{gray}{0.85}
\definecolor{dark}{gray}{0.6}

 \def\NZQ{\mathbb}               

 \def\ZZ{{\NZQ Z}}

 \def\frk{\mathfrak}               

 \def\mm{{\frk m}}
 
 \def\nn{{\frk n}}
 \def\G{{\mathcal G}}

\def\Oc{{\mathcal O}}
\def\Tc{{\mathcal T}}
 \def\Cc{{\mathcal C}}

 \def\opn#1#2{\def#1{\operatorname{#2}}} 
 \opn\chara{char} \opn\length{\ell} \opn\pd{pd} \opn\rk{rk}
 \opn\projdim{proj\,dim} \opn\injdim{inj\,dim} \opn\rank{rank}
 \opn\depth{depth} \opn\grade{grade} \opn\height{height}
 \opn\embdim{emb\,dim} \opn\codim{codim}
 
 \opn\Tr{Tr} \opn\bigrank{big\,rank}
 \opn\superheight{superheight}\opn\lcm{lcm}
 \opn\trdeg{tr\,deg}
 \opn\reg{reg} \opn\lreg{lreg} \opn\ini{in} \opn\lpd{lpd}
 \opn\size{size} \opn\sdepth{sdepth}
 \opn\link{link}\opn\fdepth{fdepth}\opn\lex{lex}
 \opn\tr{tr}
 \opn\type{type}
 \opn\gap{gap}
 \opn\diam{diam}
 \opn\Mod{Mod}
 \opn\div{div} \opn\Div{Div} \opn\cl{cl} \opn\Cl{Cl}
 \opn\Spec{Spec} \opn\Supp{Supp} \opn\supp{supp} \opn\Sing{Sing}
 \opn\Ass{Ass} \opn\Min{Min}\opn\Mon{Mon}
 \opn\Ann{Ann} \opn\Rad{Rad} \opn\Soc{Soc}
 \opn\Im{Im} \opn\Ker{Ker} \opn\Coker{Coker} \opn\Am{Am}
 \opn\Hom{Hom} \opn\Tor{Tor} \opn\Ext{Ext} \opn\End{End}
 \opn\Aut{Aut} \opn\id{id}
 
 \opn\nat{nat}
 \opn\pff{pf}
 \opn\Pf{Pf} \opn\GL{GL} \opn\SL{SL} \opn\mod{mod} \opn\ord{ord}
 \opn\Gin{Gin} \opn\Hilb{Hilb}\opn\sort{sort}
 \opn\PF{PF}\opn\Ap{Ap}
 \opn\dist{dist}
 \opn\aff{aff}
 \opn\relint{relint} \opn\st{st}
 \opn\lk{lk} \opn\cn{cn} \opn\core{core} \opn\vol{vol}  \opn\inp{inp} \opn\nilpot{nilpot}
 \opn\link{link} \opn\star{star}\opn\lex{lex}\opn\set{set}
 \opn\width{wd}
 \opn\Fr{F}
 \opn\QF{QF}
 \opn\G{G}
 \opn\type{type}\opn\res{res}
 \opn\conv{conv}
 \opn\subtr{subtr}
 \opn\gr{gr}
 
 \def\pot#1#2{#1[\kern-0.28ex[#2]\kern-0.28ex]}

 \opn\dirlim{\underrightarrow{\lim}}
 \opn\inivlim{\underleftarrow{\lim}}
 \let\sect=\cap
 \let\dirsum=\oplus
 
 \let\iso=\cong
 
 \let\Sect=\bigcap

 \let\to=\rightarrow
 
 \def\Implies{\ifmmode\Longrightarrow \else
         \unskip${}\Longrightarrow{}$\ignorespaces\fi}
 \def\implies{\ifmmode\Rightarrow \else
         \unskip${}\Rightarrow{}$\ignorespaces\fi}
 \def\iff{\ifmmode\Longleftrightarrow \else
         \unskip${}\Longleftrightarrow{}$\ignorespaces\fi}

 \let\:=\colon
 \newtheorem{Theorem}{Theorem}[section]
 \newtheorem{Lemma}[Theorem]{Lemma}
 \newtheorem{Corollary}[Theorem]{Corollary}
 \newtheorem{Proposition}[Theorem]{Proposition}

 \newtheorem{Example}[Theorem]{Example}
 \newtheorem{Examples}[Theorem]{Examples}

 \newtheorem{Question}[Theorem]{Question}
 
 \let\epsilon\varepsilon
 \let\kappa=\varkappa
 \def\qed{\ifhmode\textqed\fi
       \ifmmode\ifinner\quad\qedsymbol\else\dispqed\fi\fi}
 \def\textqed{\unskip\nobreak\penalty50
        \hskip2em\hbox{}\nobreak\hfil\qedsymbol
        \parfillskip=0pt \finalhyphendemerits=0}
 \def\dispqed{\rlap{\qquad\qedsymbol}}
 
 \opn\dis{dis}
 \def\pnt{{\raise0.5mm\hbox{\large\bf.}}}
 
 \opn\Lex{Lex}

\begin{document}
\title {On the set of trace ideals of a Noetherian ring}

\author {J\"urgen Herzog and  Masoomeh Rahimbeigi}

\address{J\"urgen Herzog, Fachbereich Mathematik, Universit\"at Duisburg-Essen, Campus Essen, 45117
Essen, Germany} \email{juergen.herzog@uni-essen.de}

\address{Masoomeh Rahimbeigi, Ilam, Haft cheshmeh, Askarinia street, 69391-13111, Iran}
\email{rahimbeigi-masoome@yahoo.com}

\dedicatory{Dedicated to Professor Ernst Kunz}

\begin{abstract}
We consider trace ideals in Noetherian rings and focus our attention to one-dimensional analytically irreducible local rings. For such rings we classify those Gorenstein rings which admit only a finite number of trace ideals.
\end{abstract}

\thanks{}




\maketitle

\setcounter{tocdepth}{1}

\section*{Introduction}
Let $R$ be a Noetherian ring, $M$ a finitely generated $R$-module, and $\varphi: M\to R$ an $R$-module homomorphism. Then  $\varphi(M)$ is an ideal in $R$. The sum of all such ideals is called the trace of $M$ and is denoted by $\tr(M)$. We call an ideal $I\subseteq R$ a {\em trace ideal} if there exists a finitely generated $R$-module $M$ such that $I=\tr(M)$.  The concept of trace is a classical notion which appears in various contexts. The trace of a module,   which  we consider here,  is not related to the trace map of an endomorphism which  plays an important role in algebraic geometry, see \cite[Appendix H]{Ku2}.   Very recently, the trace of the canonical module  has led to the  notion of nearly Gorenstein rings, see \cite{HHS} and \cite{HHS1}. Of course in general  not all ideals are trace ideals. Lindo and Pande \cite{LP}  proved that a local ring is an artinian Gorenstein ring if and only if every ideal is a trace ideal,  and Kobayashi and Takahashi \cite{KT} classified all Noetherian  local rings  for which every ideal is isomorphic to a trace ideal.

In this paper we focus on the problem to identify trace ideals in $R$. In the first section we recall a few basic properties of  trace ideals which are  mostly borrowed from the paper of Lindo  \cite{Li}. The following important facts  help us to identify the trace ideals.   Namely, an ideal $I$ is  a trace ideal if and only if $I=\tr(I)$. One also  has $\tr(I)=II^{-1}$   for any ideal $I\subsetneq R$ with  $\grade(I)>0$. Here $I^{-1}$ is the set of elements $f$ in the full ring of quotients $Q(R)$ of $R$ with the property that $fI\subsetneq R$. We use the fact that $I$ is a trace ideal if and only if $\tr(I)=I$ to show in Proposition~\ref{gorensteindec22} that a generically Gorenstein ring with canonical ideal $\omega_R$ is Gorenstein if and only if $\omega_R\iso \tr(\omega_R)$.

Another interesting fact is that the sum of trace ideals is again a trace ideal. This has the consequence   that for any given ideal $I$ there exists a unique, with respect to inclusion, maximal  trace ideal contained in $I$ which we denote by $\subtr(I)$, and we have  $\subtr(I)\subseteq I\subseteq \tr(I)$. In general there is no  smallest trace ideal containing $I$. In particular, $\tr(I)$ may not be  the smallest trace ideal containing $I$. A smallest trace ideal containing $I$ would exist if the intersection of trace ideal is again a trace ideal, which also fails to be true in general. We close Section~1 by showing in Theorem~\ref{final} that in a local Cohen--Macaulay domain $R$ the set of trace ideals is precisely the set of ideal of height $>1$,  if and only if $R$ is normal. In Theorem~\ref{final} the hypothesis that $R$ is a domain is required. Indeed, we give an example of a  trace  ideal $I\subsetneq R$  of height 1, where  $R$ is Cohen--Macaulay and normal, but not a domain.  These results leave the question open for which general local rings  an ideal $I$  is a trace ideal if and only if $\height(I)>1$.

In Section~2 we study the trace ideals of one-dimensional analytically irreducible  local rings. Among them are the semigroup rings attached to numerical semigroups. If $(R,\mm)$ is a one-dimensional analytically irreducible  local ring, then its integral closure $(\overline{R}, \nn)$ is a discrete valuation ring and a finite $R$-module. We assume in addition that $R$ and $\overline{R}$ have the same residue class field. The ideal $C=R:\overline{R}$ is called the conductor of $R$. If $\nn=(t)$, then $C=t^c\overline{R}$ for some integer $c\geq 0$. Trace ideals are rigid in the sense that if $I$ and $J$ are trace ideals and $I\iso J$,  then $I=J$. We denote by $\Tc_R$ the set of trace ideals of $R$ and by $\Cc_R$ the set of  ideals $I$ with $C\subseteq I\subseteq R$. It can be seen that $\Tc_R$ is a subset of  $\Cc_R$, which in most cases is a proper subset. There is an explicit set of trace ideals. Indeed for  each $f\in R\setminus \nn C$, the ideal $f\overline{R}\sect R$ is a trace ideal. These are exactly the integrally closed  ideals which contain the conductor.  Let  $v$ denote the valuation induced by $\overline{R}$, and $H=v(R)$ the value semigroup of $R$.  There exist exactly $n(H)+1$ inegrally  ideals containing the conductor, where $n(H)$ denotes the number of elements $h\in H$ with $h<c$.

We denote by $\Oc_R$ the set of overrings  $R'$ of $R$ which are contained in $\overline{R}$. For the rest of the paper Proposition~\ref{gorensteinisdifficult} is crucial.  This result  is   due to Goto, Isobe and  Kumashiro \cite[Corollary 2.8]{GIK}, and it  says that if $R$ is Gorenstein, then  the map $\alpha\: \Tc_R\to \Oc_R$, $I\mapsto  I^{-1}$ is bijective. By using this fact and assuming that $R$ is Gorenstein,  the maximal trace ideal contained in $I$ can be computed for an ideal $I\in  \Cc_R$  as follows: there is a smallest integer $k_0$ such that $(I^{-1})^{k_{0}}=(I^{-1})^{k}$ for all $k\geq k_0$. Then $\subtr(I)=J^{-1}$ where $J=(I^{-1})^{k_{0}}$. Another application is Corollary~\ref{ProfessorKunz} which asserts  that if  $R$ is Gorenstein and $|R/\mm|=\infty$ and $H$ is the value semigroup of $R$,  then the following conditions are equivalent: (i) $R$ has only  finitely many trace ideals, (ii)  $H=\langle 2,a\rangle$ where $a>2$ is an odd number, or $H=\langle 3,4\rangle$ or $H=\langle 3,5\rangle$, (iii) all trace ideals of $R$  are integrally closed.

In the last section we consider certain binary operations on the set of trace ideals in  one-dimensional analytically irreducible local domains. As mentioned before, intersections of trace ideals are  in general not trace ideals. The same is the case for products of trace ideals. If we replace the usual product of two ideals $I$ and $J$ by the product $I*J=(I^{-1}J^{-1})^{-1}$, then it turns out that $I*J$ is a trace ideal if $R$ is Gorenstein and if $I$ and $J$ are trace ideals. Moreover, if $R$ is Gorenstein,  this product is associative and coincides with the largest trace ideal contained in $I\sect J$. In general one has $I*(J+L)\subseteq I*J+I*L$. Equality holds if $R$ is the semigroup ring of a numerical semigroup and the ideals $I$,  $J$ and $L$ are monomial ideals. Hence in this case $\Tc_R$ has the structure of a semiring.

We gratefully acknowledge the use
of the numericalsgps package \cite{Gap}  in GAP for our computations.

\section{The trace of a module and trace ideals}

Let $R$ be a Noetherian commutative ring  and $\Mod_R$ be the category of finitely generated $R$-modules. Let $M\in \Mod_R$. The {\em trace} of $M$, denoted $\tr_R(M)$, is defined to be  the sum of the ideals $\varphi(M)$ with $\varphi \in \Hom_R(M,R)$. Thus
$$
\tr_R(M)=\sum_{\varphi \in \Hom_R(M,R)}\varphi(M).
$$
For simplicity we write $\tr(M)$ instead of $\tr_R(M)$, if no confusion is possible.

It follows immediately from the definition of the trace that $\tr(M)=\tr(N)$ if $M\iso N$.

We call an ideal $I\subsetneq R$ a {\em trace ideal}, if there exists  $M\in\Mod_R$ with $I=\tr(M)$.

Lindo~\cite[Proposition 2.8(iv)]{Li} observed  that trace  ideals are characterized as follows.

\begin{Proposition}
\label{doormakesproblems}
Let $I\subsetneq R$ be an ideal. Then $I\subseteq \tr(I)$, and $I=\tr(I)$ if and only if $I$ is a trace ideal.
\end{Proposition}

\begin{Corollary}
\label{notallispossible}
Let $I, J\subsetneq R$   be  non-zero  ideals. Then the following holds:
\begin{enumerate}
\item [(a)]  Let $I$ and $J$ be trace ideals. Then $I$ is isomorphic to $J$ if and only if $I=J$.
\item[(b)]  Suppose $I\neq (0)$ and $I=fJ$,  where $f$ is a non zero-divisor which is contained in all maximal ideals of $R$. Then $I$ is not a trace ideal.
\end{enumerate}
\end{Corollary}
\begin{proof}
(a) Suppose $I\iso J$. Then
$
I=\tr(I)=\tr(J)=J.
$

(b) Since $I\iso J$ it follows that $\tr(I)=\tr(J)$. Suppose $I$ is a trace ideal. Then
$fJ=I=\tr(I)=\tr(J)\supseteq J$. Therefore, $fJ=J$, which implies that $f^2J=fJ=J$. Thus,  by induction, $f^kJ=J$ for all $k$, and hence $J=\Sect_{k\geq 0}f^kJ=(0)$, by Krull's intersection theorem, see \cite[Chapter V, Corollary 5.7]{Ku}. This is  a contradiction since $J\neq (0)$.
\end{proof}

In the next section we  give an example which shows that the converse of Corollary~\ref{notallispossible}(b) is not true in general, see Example~\ref{angelhasmanyfriends}.

\medskip
Products of trace ideals need not  be trace ideals. But as shown in \cite[Proposition 2.8(ii)]{Li} one has

\begin{Proposition}
\label{whogetsvaccinefirst}
The sum of  trace ideals is again a trace ideal. Indeed, $$\tr(M\dirsum N) =\tr(M)+\tr(N) \quad \text{for} \quad M,N\in\Mod_R.$$
\end{Proposition}

\begin{Corollary}
\label{mytimeisover}
 Let $I\subseteq R$ be an ideal. Then there exists a  unique trace ideal contained in $I$ which is maximal with respect to inclusion, which we denoted by $\subtr(I)$. Hence
 \[
 \subtr(I)\subseteq I\subseteq  \tr(I),
 \]
and $\subtr(I)\subsetneq I$ if and only if  $I\subsetneq \tr(I)$.
\end{Corollary}

\begin{proof}
The ideal $\subtr(I)$ is just  the  sum of all trace ideals which are contained in $I$. This sum is not empty, because $(0)\subseteq I$ is a trace ideal. Moreover, $\subtr(I)\subseteq I$.  Since $R$ is Noetherian, this sum is a finite sum. Therefore,  by Proposition~\ref{whogetsvaccinefirst}, $\subtr(I)$ is a trace ideal. By its construction,  $\subtr(I)$ is the largest trace ideal contained in $I$.
\end{proof}

\begin{Example}
{\em
Let  $K$ be a field, $R=K[|t^3,t^4|]$ and let $I=(t^3,t^8)$. Then $\tr(I)=(t^3,t^4)$ which is the  maximal ideal of $R$. Since $I\neq \tr(I)$, $I$ is not a trace ideal. On the other hand, the ideal $J=(t^6,t^7,t^8)$ is a trace ideal. Note that  $I/J\iso R/\mm$. Therefore, any ideal $L\subseteq I$ with $J\subsetneq L$ is equal to $I$. This shows that $J$ is the biggest ideal contained in $I$ which is different from $I$. Hence,  $J=\subtr(I)$.  So, in this case we have
\[
\subtr(I)=(t^6,t^7,t^8)\subsetneq (t^3,t^8)\subsetneq (t^3,t^4)=\tr(I).
\]}
\end{Example}

Similarly  to \cite[Proposition 2.8(iii)]{Li}, we have
\begin{Proposition}
\label{manypossibilities}
Let  $R$ be a local ring  or a positively graded $K$-algebra where $K$ is a field, and let $M\in\Mod_R$. When $R$ is graded we assume that $M$ is also graded.  Under these assumptions  the following conditions are equivalent:
\begin{enumerate}
\item[(i)] $\tr(M)=R$.
\item[(ii)] There exists a  module $N$ such that $M\iso N\dirsum R$.
\end{enumerate}
\end{Proposition}

\begin{proof}
(i) \implies (ii): Since $\tr_R(M)=\sum_{\varphi \in \Hom_R(M,R)}\varphi(M)=R$,  our assumption on  the ring $R$ implies that  there exists $\varphi\: M\to R$ with $\varphi(M)=R$.  Therefore $\varphi\: M\to R$ is surjective. Since $R$ is a free $R$-module, the map $\varphi$  splits. This implies that $M\iso \Ker(\varphi)\dirsum R$.

(ii) \implies (i): Since $M\iso N\dirsum R$, it follows from Proposition~\ref{whogetsvaccinefirst} that $\tr(M)=\tr(N\dirsum R)=\tr(N)+\tr(R)=R$. This shows that $\tr(M)=R$.
\end{proof}
Let $Q(R)$ be a full ring of fractions of $R$. Then for any ideal $I$ with $\grade(I)\geq1$,  one has the well-known fact that
\begin{eqnarray}
\label{traceofideal}
\tr(I)=II^{-1},
\end{eqnarray}
where $I^{-1}= \{f \in Q(R)\:\; fI\in R\}$, see  \cite[Proposition 2.8(iii)]{Li}. Therefore,  if $\grade(I)\geq 1$, Proposition~\ref{doormakesproblems} implies that  $I$ is  a trace ideal if and only if $I=II^{-1}$.

\medskip
 Observe that $I^{-1}$ is a fractionary ideal. Recall that  any  finitely generated  $R$-submodule  of $Q(R)$ is called  a {\em fractionary ideal}. In other words, if $I$ is a fractionary ideal, then  there exist
$f_1,\ldots,f_r\in Q(R)$ such that $I=(f_1,\ldots,f_r)$. We have   $f_i=h_i/g_i$ with $g_i,h_i\in R$, where   $g_i$ is a non zero-divisor. It follows that $gI$ is an ideal in $R$ for $g=g_1g_2\cdots g_r$ and shows that,  as an $R$-module, any fractionary ideal is isomorphic to an ideal of $R$.

We use the identity (\ref{traceofideal}) to give the following characterization of Gorenstein rings.

\begin{Proposition}
\label{gorensteindec22}
Let $R$ be a local Cohen--Macaulay ring  and suppose that $R$ admits a canonical module $\omega_R$. Then $R$ is Gorenstein if and only if $\omega_R\iso \tr(\omega_R)$.
\end{Proposition}

\begin{proof}
If $R$ is Gorenstein, then $\omega_R\iso R$, and hence $\tr(\omega_R)=\tr(R)=R=\omega_R$.

Conversely, suppose that $\omega_R\iso \tr(\omega_R)$. Let $P$ be a minimial prime ideal of $R$. Then $\omega_{R_P}\iso (\omega_R)_P\iso  \tr(\omega_R)_P\iso \tr(\omega_{R_P})$. Therefore, considering the length of the modules,  we get  $\length(\tr(\omega_{R_P}))=\length(\omega_{R_P})=\length(R_P)$, where the last equality follows from Matlis duality.  This implies that $\tr(\omega_{R_P})=R_P$. Thus $R_P$ is Gorenstein. Thus, our hypothesis that $\omega_R\iso \tr(\omega_R)$ implies that $R$ is generically Gorenstein.

If $\dim R=0$, then $R$ is Gorenstein and hence $\omega_R\iso R$ and there is nothing to prove. Suppose now that $\dim R>0$. Since $R$ is generically Gorenstein, it follows that  $\omega_R$ is isomorphic to a fractionary ideal, see \cite[Korollar~6.7]{HK1} or \cite[Proposition 3.3.18]{BH}. Thus, our assumption implies that $\omega_R\iso \omega_R\omega_R^{-1}$. Then
\[
R=\omega_R:\omega_R\iso\omega_R: (\omega_R\omega_R^{-1})=(\omega_R:\omega_R):\omega_R^{-1}=R:\omega_R^{-1}=(\omega_R^{-1})^{-1}.
\]
The first equality follows from \cite[Satz 6.1]{HK1} or \cite[Theorem 3.3.10]{BH}. Taking once again the inverse on both sides of this equation, we find that $R=((\omega_R^{-1})^{-1})^{-1}=\omega_R^{-1}$. For the last equation we use
 the fact  that $I^{-1}$ is reflexive for all ideals $I$ containing a non zero-divisor. Now since $\omega_R^{-1}$ is a principal ideal, \cite[Korollar 7.33]{HK1} implies that $\omega_R$ itself is a principal ideal, which in turn implies that $R$ is Gorenstein.
\end{proof}

\begin{Proposition}
\label{maximalideal}
Let $K$ be  a field, and  let  $R$ be a local ring or a positively graded $K$-algebra. Furthermore, let  $\mm$ be  the maximal (graded) ideal of $R$.  If $\mm$  is not a principal ideal, then $\mm$ is a trace ideal.
\end{Proposition}

\begin{proof}
Suppose $\mm$ is not a trace ideal. Then $\tr(\mm)\neq \mm$. Since $\mm\subsetneq \tr(\mm)$,  it follows that $\tr(\mm)=R$. Then, by Proposition~\ref{manypossibilities}, we have $\mm\iso N\dirsum R$ for some module $N$.
This means that there exists an ideal $I\subsetneq\mm$ and  a non zero-divisor $f\in \mm$ such that $\mm= I+(f)$ and $I\sect (f)=(0)$. Suppose $I\neq (0)$. Then there exists $0\neq g\in I$. Since $f$ is a non zero-divisor $fg\neq 0$. Since $fg\in I\sect (f)$ it follows that $I\sect (f)\neq(0)$, a contradiction. Hence $\mm=(f)$ and this means $\mm$ is a principal ideal.
\end{proof}

 Let $R=K[x]/(x^2)$. Then the maximal ideal of $R$ is a principle ideal and also a trace ideal. This shows that the converse of Proposition~\ref{maximalideal} is  in general not correct.

\medskip

Finally we recall the following result \cite[Example 2.4]{Li}.

\begin{Proposition}
\label{grade}
Suppose $I\subsetneq R$ is an ideal with $\grade(I)>1$. Then $I$ is a trace ideal. In particular, if $R$ is  Cohen-Macaulay  and  $\height(I)>1$, then $I$ is a trace ideal.
\end{Proposition}

 Now we show

\begin{Theorem}
\label{final}
 Let $R$ be a Cohen--Macaulay domain. Then the following conditions are equivalent:
 \begin{enumerate}
 \item[(i)] $I$ is a trace ideal of $R$ if and only if $\grade(I)>1$.
 \item[(ii)] $R$ is normal.
 \end{enumerate}
\end{Theorem}
\begin{proof}
(i)\implies (ii): By our assumption, if $I$ is an ideal with $\grade(I)=1$,  then $I$ is not a trace ideal. Hence, if $P$ is a prime ideal of height one, it follows that $P\neq \tr(P)$. Therefore, $P\subsetneq \tr(P)$. Hence,  after localization $R_P=\tr(P)R_P=\tr(PR_P)$. Here we use that trace localizes.  Since $PR_P\neq R_P$,  this shows that $PR_P$ is not a trace ideal. By Proposition~\ref{maximalideal}, $PR_P$ is a principal ideal. Therefore $R_P$ is a discrete valuation ring. It follows  that $R$ satisfies the Serre condition $R_1$.    By  Serre's  normality criterion \cite[ § 23]{Mats} this implies that $R$ is normal.

(ii)\implies (i): Suppose there exists  a trace ideal $I$ with $\grade(I)=1$. Then,  since $R$ is Cohen--Macaulay, $\height(I)=1$,  and hence there exists a prime ideal $P$ of height one with $I\subseteq P$. Now since $I$ is a trace ideal we have $I=\tr(I)$ and hence $IR_P=\tr(IR_P)$. Therefore, $IR_P$ is a trace ideal of $R_P$ which is different from $R_P$.  By assumption, $R$ is normal.  Therefore $R_P$ is a discrete valuation. This  implies that  $IR_P$ is a  principal ideal generated by non a zero-divisor, contradicting Corollary~\ref{notallispossible}(b).
\end{proof}
\begin{Example}
\label{thenicestfaceintheworld}
{\em
Let $K$ be a field, $S=K[x,y]$ the polynomial ring over $K$ with the variables  $x$ and $y$ and $R=S/(xy)$. Then $R$ is a one dimensional Cohen--Macaulay ring. Also  $R$ is normal because it is the direct product of normal domains, see \cite[§ 23]{Mats}. Indeed, $R\iso K[x]\times K[y]$.

Set $\mm=(x,y)$. Next we claim that $\tr(\mm)=\mm.$ Since $\grade(\mm)=1$ and $R$ is normal, this example shows that in Theorem~\ref{final} we cannot drop the hypothesis that $R$ is a domain.

In order to prove the claim we assume to the contrary that $\tr(\mm)\neq \mm$. Then $\tr(\mm)=R$ and hence there exists a surjective $R$-module homomorphism $\epsilon\:\mm\to R$. Let $Z$ be the kernel of this epimorphism and let $t=x+y$. Then $t\in \mm$ is a non zero-divisor and hence $\mm R_t=R_t$. Here $R_t$ is the ring which is obtained from $R$ by localization with respect to powers of $t$. It follows that $\varepsilon$ after localization becomes an isomorphism. Therefore, $ZR_t=0$. This implies that $t^kZ=0$ for some $k$. Hence $Z$ is a  $R/(t^k)$--module. Since $\dim R=1$ and since $t$ is a non zero-divisor  it follows that $\dim R/(t^k)=0$, and so, $\dim Z=0$. In particular, $\depth Z=0$. Since $Z$ is a submodule of $\mm$, $\depth \mm=0$. This is a contradiction, because $\depth \mm=1$.
}
\end{Example}

\begin{Question}

{\em
Which are the rings with the property that an ideal $I$ is a trace ideal if and only if $\height(I)>1$?}
\end{Question}

\section{Trace ideals of one-dimensional analytically irreducible local rings}

In this section we assume that $(R,\mm)$ is a $1$-dimensional domain with quotient field $Q(R)$ whose integral closure $\overline{R}\subsetneq Q(R)$ is a discrete valuation ring and a finite $R$-module. This is equivalent to saying that $R$ is analytically irreducible.   Let $v\: Q(R)\to \ZZ$ be the valuation attached to $R$. The set $v(R)=\{v(r)\: r\in R\setminus\{0\} \}$ is called the {\em value semigroup} of $R$. For details we refer to \cite{HK}.

For a non-empty subset $A\subseteq \ZZ_{\geq 0}$ we denote by $\langle A\rangle$ the smallest additively closed subset containing $A$. Let $H=\langle A\rangle$. One has   $\gcd(A)=1$ if and only if  there exists an integer $c\geq 0$ such that $H =c+\ZZ_{\geq 0}$. In this case,  $H$ is called a {\em numerical semigroup}. The elements of $A$ are called the {\em generators} $H$. Any numerical semigroup is finitely generated. The value semigroups $v(R)$ of $R$ is a numerical semigroup. The smallest element $ h\in H$ with $h\neq 0$ is called the {\em multiplicity} of $H$ which is denoted by $e(H)$.

Note that $\overline{R}$ is a finitely generated $R$-module of rank one, and hence a fractionary ideal  containing $R$. This implies that $C=R:\overline{R}$ is an ideal in $R$, which is called the {\em conductor } of $R$. We first recall  a few well-known facts.

\begin{Lemma}
\label{largest}
The conductor $C$ is the largest ideal in $R$ such that $C\overline{R}=C$.
\end{Lemma}

\begin{proof}
We first show that $C\overline{R}=C$. Indeed, let $f\in C$ and $g\in \overline{R}$. Then
\[
(fg)\overline{R}= f(g\overline{R})\subseteq f\overline{R}\subseteq R.
\]
This shows that $C\overline{R}\subseteq C$. The converse direction is trivial.
Now let $D\subseteq R$ be  an ideal such that $D\overline{R}=D$. Let $f\in D$. Then $f\overline{R}\subseteq D\overline{R}=D\subseteq R$. Therefore, $f\in R: \overline{R}=C$. This shows that $ D\subseteq C$.
\end{proof}

\begin{Proposition}
\label{conductor}
Let $I\neq 0$ be a trace ideal of $R$. Then $C\subseteq I\subseteq R$.
\end{Proposition}

\begin{proof}
It is clear that trace ideals are ideals in $R$.  Next we show that $C\subseteq I$.  We have  $I=II^{-1}$ because  $I$ is a trace ideal.  There exists $f\in I$ such that $I\overline{R}=f\overline{R}$.  Now let $g\in C$. Then
 \[
 (g/f)I\subseteq (g/f)I\overline{R}=(g/f)f\overline{R}=g\overline{R} \subseteq C\subseteq R.
 \]
This shows that $g/f\in I^{-1}$. Hence $g=f(g/f)\in II^{-1}=I$, and so $C\subseteq I$.
\end{proof}

We denote by $\nn$ the maximal ideal of $\overline{R}$. Then $\nn=(t)$ for some $t\in \overline{R}$. Note that $\mm=R\sect \nn$ because $R\subseteq \overline{R}$ is an integral extension. Since the kernel of the map  $R\to \overline{R}/\nn$ is $R\sect \nn$ it follows that the induced map $R/\mm\to R/\nn$  is injective. For the rest of this paper we assume that $R/\mm\to R/\nn$ is an isomorphism. This is equivalent to saying that for each $g\in \overline{R}$,  there exists $f\in R$ such that $f-g\in\nn$.

\begin{Lemma}
\label{angelcomes}
Let $C=t^c\overline{R}$.  Then $t^{c-1}\not\in R$.
\end{Lemma}

\begin{proof}
Suppose $t^{c-1}\in R$. Then  $D=t^{c-1}R+C$ is an ideal in $R$.  We  claim that $D\overline{R} =D$. Indeed, let $h\in D$ and $g\in \overline{R}$. Since  $gC\subsetneq C$, it suffices to show that  $gt^{c-1}\in D$. There exists $f\in R$ such that $g=f+s$ with $s\in \nn$. Then
$t^{c-1}g =t^{c-1}f+t^{c-1}s$. Since $t^{c-1}f\in t^{c-1}R$ and $t^{c-1}s\in t^c\overline{R}=C$,  it follows that $hg\in D$. This proves the claim. Since $C\subsetneq D$ and $D\overline{R} =D$, this contradicts Lemma~\ref{largest}.
\end{proof}

As in Lemma~\ref{angelcomes}, let $C=t^c\overline{R}$ be the conductor of $R$ and let $H=v(R)$. Then $a\in H$ for all $a\geq c$ and $c-1\not \in H$. The elements in $\ZZ_{\geq 0}\setminus H$ are called the {\em gaps} of $H$. The largest gap of $H$ is $c-1$ and  it is called the {\em Frobenius number} of $H$. The number of gaps of $H$ is denoted by $g(H)$ and called the {\em genus} of $H$. The elements  $a\in H$ with $a<c-1$ are called the {\em non-gaps} of $H$. The number of non-gaps is denoted by $n(H)$. It is easy to see that $n(H)\leq g(H)$. If equality holds, then $H$ is called {\em symmetric}. Kunz showed \cite{Ku1} that $R$ is a Gorenstein ring if and only if $H=v(R)$ is symmetric.

\begin{Proposition}
\label{head}
Let  $I=fJ$ with $f\in\mm$ and $J\subseteq R$. Then $C\not\subseteq I$.
\end{Proposition}

\begin{proof}
We have $C=t^c\overline{R}$. Since $fJ\subseteq (f)$, it is enough to show that $C\not\subseteq (f)$. Indeed, $ft^{c-1}\notin (f)$ because of Lemma~\ref{angelcomes}.  But we have $ft^{c-1}\in C$ because $f\in \mm$.
\end{proof}

In Proposition~\ref{conductor} we have seen that $C\subseteq I\subseteq R$ if $I$ is a trace ideal.  The following  example shows that the converse of this statement is not true in general.
\begin{Example}
\label{angelhasmanyfriends}
{\em Let $R=K[| t^5,t^6,t^7|]$ and let $I= (t^6, t^{10}, t^{14})$. Then $$(t^{10},t^{11}, t^{12}, t^{13}, t^{14})=C\subseteq I\subseteq R.$$ But $I$ is not a trace ideal because $(t^6, t^{7}, t^{10})=\tr(I)\neq I$.

We also find that $I$ is not of the form $fJ$, see Proposition~\ref{head}.
 }
\end{Example}
Let $I\subsetneq R$ be an ideal. Then there exists $f\in I$ such that $I\overline{R}= f\overline{R}$. The ideal $I$ is integrally closed if and only if $I=f\overline{R}\sect R$.

\begin{Proposition}
\label{full}
Let $I$ be an integrally closed ideal of $R$ with $C\subseteq I$. Then $I$ is a trace ideal.
\end{Proposition}

\begin{proof}
Suppose  $I\subsetneq R$ is an integrally closed ideal  containg the conductor, but $I$ is not a trace ideal.  Then $I^{-1}I\neq I$.  Hence there exists $g\in I^{-1}$ such that $gI\subsetneq I$. Therefore there exists $h\in I$ such that $gh\not\in I$.  Since $I=f\overline{R}\sect R$ with $f\in I$, and $gh\in R$, it follows that $gh\not\in f\overline{R}$. This implies that $gh/f\not\in \overline{R}$. Since $\overline{R}$ is discrete valuation ring, we see that $f/gh\in \nn$. Now we use that $h\in I$ and therefore $h=fl$ with $l\in \overline{R}$. Then we conclude that $1/gl\in \nn$ which implies that $1/g\in \nn$, too. Hence $t^{c-1} (1/g)\in C\subseteq I$. It follows that $t^{c-1}= gt^{c-1} (1/g)\in I^{-1}I\subseteq R$, a contradiction to Lemma~\ref{angelcomes}.
\end{proof}

\begin{Corollary}
\label{khazaghestan}
Let $H$ be the value semigroup of $R$. Then the number of trace ideals is at least $n(H)+1$.
\end{Corollary}

\begin{proof}
Let $a$ be a non-gap of $H$. Then there exists an element  $f\in R$ with $v(f)=a$. It follows that $f\overline{R}=t^a\overline{R}$. The  integrally closed  ideal $I_a=f\overline{R}\cap R$ is a trace ideal. If $a'$ is a non-gap with $a'\neq a $, then $I_a\neq I_{a'}$. This shows that together with $C$,  which is also a trace ideal, we have at least $n(H)+1$ trace ideals.
\end{proof}

Let $I\subseteq R$ be a non-zero ideal. Then $I\subseteq (I^{-1})^{-1}$. $I$ is called reflexive if $I= (I^{-1})^{-1}$.

For the proof of next results  we need two lemmata.  The first  lemma can be found in \cite{Be}.

\begin{Lemma}
\label{moneyback} {\em (a)}  Let  $I\subseteq R$ be a non-zero ideal. Then $I^{-1}$ is reflexive.

{\em (b)} $R$ is Gorenstein if and only if all fractionary ideals of $R$ are reflexive.
\end{Lemma}

It follows from Lemma~\ref{moneyback}(a) that $(I^{-1})^{-1}$ is reflexive. In fact, $(I^{-1})^{-1}$ is a smallest reflexive ideal containing $I$. Therefore,  $(I^{-1})^{-1}$ is called the {\em reflexive hull} of $I$.
Furthermore,  Lemma~\ref{moneyback}(b) implies that   if  $I$ and $J$ are fractionary ideals in a Gorenstein ring,  then $I=J$ if and only if $I^{-1}=J^{-1}$.

\medskip

Let $I$ be a fractionary ideal of $R$. Then $\Hom_R(I,I)$ is the endomorphism ring of $I$ which we denote by  $\End_R(I)$. It follows that   $\End_R(I)=I:I$, because $\Hom_R(I,J)=J:I$ for any two fractionary ideals $I$ and $J$ of $R$. We recall the following well-known facts.

\begin{Lemma}
\label{nicefriends}
Let $I$ be a fractionary ideal. Then
\begin{enumerate}
\item[(a)]  $R\subseteq \End_R (I)\subseteq \overline{R}$.
\item[(b)]  If $R'$ is a ring with $R\subseteq R' \subseteq \overline{R}$, then there exists a fractionary ideal $I$ such that $R'=\End_R(I)$.
\item[(c)]  If $R'$ is a ring with $R\subseteq R' \subseteq \overline{R}$, then $R'$ is a fractionary ideal and $R'=\End_R(R')$.
\end{enumerate}
\end{Lemma}
\begin{proof}

(a) Since $I$ is a fractionary ideal, it follows that $RI\subseteq I$. This shows that $R\subseteq \End_R(I)$. Now let $f\in \End_R(I)$, then $fI\overline{R}\subseteq I\overline{R}$. Since $I\overline{R}= (t^a)\overline{R}$ for some $a$, it follows that  $ft^a\overline{R}\subseteq t^a\overline{R}$, which implies that $f\overline{R}\subseteq \overline{R}$. Since $1\in \overline{R}$, we see that  $f=f1\in \overline{R}$.

(c) \implies (b) is clear.  For the proof of (c) note that $R'$ is a fractionary ideal, because $RR'\subseteq R'$. Since $R'$ is a ring, we have $R'R'=R'$, and hence $R'\subseteq \End_R(R')$. Conversely, let $f\in \End_R(R')$.  Then $f=f1\in R'$. This shows that $\End_R(R')\subseteq R'$.
\end{proof}

 Let $\mathcal{F}_R$ be the set of non-zero  fractionary ideals of $R$. We define  subsets of $\mathcal{F}_R$ as follows:  $\Cc_R=\{ I\in\mathcal{F}_R\: \; C\subseteq I\subseteq R\}$, $\Tc_R=\{I\in\mathcal{F}_R\: \; \text{$I$ is a trace ideal of $R$}\}$,  and $\Oc_R=\{ I\in\mathcal{F}_R\: \; \text{$I$ is an overring of $R$}\}$.

For the proof of the main result of this section we need the following fact, see  \cite[Corollary 2.8]{GIK}.
\begin{Proposition}
\label{gorensteinisdifficult}
Let $R$ be Gorenstein, and let $\alpha\:\Tc_R\to \mathcal{F}_R$ be the map with $I\mapsto I^{-1}$. Then $\Im\alpha\subseteq \Oc_R$ and  $\alpha\: \Tc_R\to \Oc_R$ is bijective.
\end{Proposition}
\begin{proof}
Since $I\in \Tc_R$, we have $I=II^{-1}$. Therefore, $\alpha(I)=I^{-1}=(II^{-1})^{-1}=I^{-1}:I^{-1}$. By Lemma~\ref{nicefriends}(a),  $I^{-1}:I^{-1}\in \Oc_R$.

$\alpha$ is injective:  Suppose $I^{-1}=J^{-1}$. Since $R$ is Gorenstein, $I=(I^{-1})^{-1}=(J^{-1})^{-1}=J$, see Lemma~\ref{moneyback}(b).

$\alpha$ is surjective: Let $R'\in \Oc_R$. Then $R'=L$, where $L$ is a fractionary ideal and $R'=L:L$, see Lemma~\ref{nicefriends}(b). Let $I=L^{-1}$. Then $I^{-1}=(L^{-1})^{-1}=L$ and $R'=I^{-1}:I^{-1}$. Since $(II^{-1})^{-1}=I^{-1}:I^{-1}=R'$ and since $II^{-1}\in \Tc_R$, the assertion  follows.
\end{proof}

In general,  not all trace ideals are reflexive.

\begin{Examples}{\em
(a) Let $R=K[| t^{14},t^{15},t^{20},t^{21},t^{25} |]\subsetneq \overline{R}=K[|t|]$. Then $I=(t^{21},t^{28},t^{29},t^{30},t^{34})$ is a trace ideal which is not reflexive.  Indeed,  $$(I^{-1})^{-1}=(t^{21},t^{28},t^{29},t^{30},t^{34},t^{39},t^{40}).$$

(b) Let $R=K[| t^3,t^4,t^5|]\subsetneq \overline{R}=K[|t|]$.  In \cite[Example 3.4(2)]{GIK}, it is  shown  that $\alpha:\Tc_R\to \Oc_R$ is not bijective. \cite[Proposition 5.1(3),(4)]{GIK} shows that $\alpha$ may be bijective, even when $R$ is not Gorenstein.}
\end{Examples}

For Gorenstein rings the  next result provides an efficient method to compute the largest trace ideal contained in a given ideal.

\begin{Corollary}
\label{highenough}
Let $R$ be Gorenstein and  $I\in \Cc_R$.  Then  there exists an integer $k_0$ such that $(I^{-1})^{k_0} =(I^{-1})^{k}$ for all $k\geq k_0$. Let $J=(I^{-1})^{k_0}$.  Then  $\subtr(I)=J^{-1}$.
\end{Corollary}

\begin{proof}
If follows from Proposition~\ref{gorensteinisdifficult} that $\subtr(I)^{-1}$ is the smallest overring in $\Oc_R$ which contains $I^{-1}$. It is clear that any overring containing $I^{-1}$ also contains $(I^{-1})^k$ for all $k\geq 0$. Since $I^{-1} \subseteq (I^{-1})^2\subseteq \ldots$  and since $R$ is Noetherian, this chain of inclusions stabilizes
 and  the desired result follows.
\end{proof}

\begin{Example}
{\em Let $R=K[|t^7,t^{10}|]$. Then $R$ is Gorenstein. The ideal  $I=(t^{14},t^{50})$  belongs to $\Cc_R$, but $I$ is not a trace ideal. Indeed, $\tr(I)=(t^{14}, t^{20})$. We have  $I^{-1}=(1,t^6)$, $(I^{-1})^2 =(1,t^6,t^{12})$ and $(I^{-1})^k = (I^{-1})^3=(1,t^6,t^{12}, t^{18})$ for all $k\geq 3$. Therefore, $\subtr(I)=((I^{-1})^3)^{-1}= (t^{38},t^{42}, t^{44}, t^{50})$. }
\end{Example}

In general $R$ may have  infinitely many trace  ideals. The following example is given in \cite[Example 3.4(1)]{GIK}:
Let $R=K[| t^4,t^5|]\subsetneq \overline{R}=K[|t|]$, and for $a\in K$ let $R_a=K[| t^4,t^5,t^6+at^7|]$. Note that $R$ is Gorenstein and  $R\subsetneq R_a\subsetneq \overline{R}=K[|t|]$.  It is  shown in \cite{GIK} that  $R_a\neq R_{a'}$ for $a\neq a'$. Therefore, Proposition~\ref{gorensteinisdifficult} implies that $|\Tc_R|=\infty$  if $|K|=\infty$.

Our next goal is to characterize all  analytically irreducible  Gorenstein rings $R$ for which  $\Tc_R$ is finite. First we observe
\begin{Lemma}
\label{earinlefthand}
Let $R_1,R_2\in \Oc_R$, and suppose that $R_1\subseteq R_2$. Then $R_1=R_2$ if and only if $v(R_1)=v(R_2)$.
\end{Lemma}
\begin{proof}
It is enough to show that $R_1=R_2$ if  $v(R_1)=v(R_2)$. Let $t^c\overline{R}$ be the conductor of $R$.  If $f\in R_2$ and $v(f)\geq c$, then $f\in \overline{R}$ and hence $f\in R_1$.   Suppose $R_1\subsetneq R_2$. Then there exists $f\in R_2\setminus R_1$.  It follows that $v(f)< c$. We may choose $f\in R_2\setminus R_1$ with $v(f)$ maximal. Since $v(R_1)=v(R_2)$, we find $g\in R_1$ such that $v(f)=v(g)$. Then there exists a unit $\varepsilon\in \overline{R}$ such that $f=\varepsilon g$. Since $R/\mm=\overline{R}/\nn$, we have  $\eta+\mm=\varepsilon+\nn$ for some  $\eta\in R$.  This implies that $\varepsilon-\eta\in\nn$, and we get
\[
f-\eta g=f-\varepsilon g+(\varepsilon -\eta )g=(\varepsilon -\eta )g.
\]
Since $\varepsilon-\eta\in\nn$, it follows that $v(f-\eta g)=v((\varepsilon -\eta )g)>v(g)=v(f)$. Then $f-\eta g\in R_1$,  by the choice of $f$.  Since $\eta \in R\subsetneq R_1$ and $g\in R_1$, this implies that $f\in R_1$, a contradiction.
\end{proof}

\begin{Theorem}
\label{ProfessorKunzlikestohelpsus}
Assume that $|R/\mm|=\infty$. Let $H$ be the value semigroup of $R$. Then the following conditions are equivalent:
\begin{enumerate}
\item[(i)] $|\Oc_R|<\infty$.
\item[(ii)]$H=\langle 1\rangle$ or $H=\langle 2,a\rangle$ where $a>2$ is an odd number,  or $H$ is one of the semigroups: $\langle 3,4\rangle$,  $\langle 3,4, 5\rangle$, $\langle 3,5\rangle$ or $\langle 3,5,7\rangle$.
\end{enumerate}
\end{Theorem}
\begin{proof}
(i)\implies (ii):  Suppose $H$ is not one of the semigroups in (ii). Then two cases are possible:

Case 1: $e(H)\geq 4$

Case 2: $e(H)=3$ and if $a$ is the smallest integer  in $H$ with $a>3$, then $a>4,5$.

 For $r\in R$ we set $\overline{r}=r+\mm$. Then $r$ is a unit in $R$ if and only if $\overline{r}\neq 0$. Since $| R/\mm|=\infty$,  for each  $i\geq 1$ there exists  $r_i\in R$ such that $\overline{r_i} \neq \overline{r_j} $ for $i\neq j$. Then $\overline{r_i-r_j}=\overline{r_i}-\overline{r_j}\neq 0$. Therefore $r_i-r_j$ is a unit in $R$ for all $i\neq j$.

Now we discuss Case 1:
 Since $e(H)\geq 4$, it follows that $2,3\not\in H$. Therefore by Lemma~\ref{earinlefthand}, $t^2,t^3\not\in R$. Now for each $i$, let  $f_i=t^2+r_it^3$ and $R_i=R[f_i]$. Then $R\subsetneq R_i\subsetneq\overline{R}$.

 We claim that $R_i\neq R_j$ for $i\neq j$. This will be a contradiction to $(i)$. Assume that  $R_i= R_j$ for some $i\neq j$. Then $f_j\in R_i$.  Therefore $f_i-f_j=(r_i-r_j)t^3$. Since $r_i-r_j$ is a unit, it follows that $t^3\in R_i$. Hence, there exists an integer $m\geq 1$ such that  $t^3=g_0+g_1f_i+g_2f_i^2+\cdots +g_mf_i^m$ with  $g_i\in R_i$.  So we may write  $t^3=g_0+g_1f_i+g$ with $g\in R_i$ and $v(g)\geq 4$. Suppose that $g_0\not\in\mm$ then $g_0\not\in\nn$. Therefore $g_0\not\in R_i\sect \nn=\mm_i$  which is the maximal ideal of $R_i$, and this is implies that  $g_0$ is a unit in $R_i$. Then $g_1f_i+g=t^3-g_0$ and  $t^3-g_0$ is a unit. This is a contradiction, because $v(g_1f_i+g)\geq \min\{v(g_1f_i), v(g)\}\geq 2$ and $v(t^3-g_0)=0$.  Hence $g_0\in \mm$, and then $v(g_0)\geq 4$.  Let  $t^3=g_1f_i+h $ where $h=g-g_0$.  Then $v(h)\geq4$. It follows that,  $v(g_1)+2=v(g_1f_i)= v(t^3-h)=3$. Therefore $v(g_1)=1$, a contradiction.

 For the proof of Case 2, we set $f_i=t^4+r_it^5$. Again we claim  that $R_i\neq R_j$ for $i\neq j$. Suppose that  $R_i= R_j$ for some $i\neq j$, then $f_j\in R_i$. Therefore, as in the Case 1, $t^5\in R_i$ and $t^5=g_0+g_1f_i+g$ where $g_0, g_1\in R$ and $g\in R_i$ with $v(g)\geq 8$. As in case 1, $g_0$ cannot be a unit. Therefore, $v(g_0)\geq 3$.  If $v(g_0)>3$ then $ v(g_0)\geq 6$ and $5=v(g_1f_i)=v(g_1)+4$. Hence $v(g_1)=1$, a contradiction. If $v(g_0)=3$, then it follows that $5=v(t^5)=v(g_0+g_1f_i+g)=3$,  again a contradiction.

 (ii)\implies(i): If $H=\langle 1\rangle$, then $R=\overline{R}$ and the assertion is trivial.  Next we consider the case $H=\langle 2,a\rangle$ where $a>2$ is an odd number. Let $R'\in \Oc_R$ with $R'\neq R$. Then  $R\subsetneq R'\subseteq \overline{R}$. By Lemma~\ref{earinlefthand} we have $v(R)\subsetneq v(R')$. Therefore,   we conclude that $v(R')=\langle 2,b\rangle$ where $b$ is odd and $b<a$. We claim that $R'=R[t^b]$. Suppose $t^b\not\in R'$. Then $R'\subsetneq R'[t^b]$. Again by using  Lemma~\ref{earinlefthand}, $v(R')\subsetneq v(R'[t^b])$. On the other hand, each element in $R'[t^b]\setminus R'$ has a value bigger than $b$ and hence $v(R'[t^b])=\langle 2,b\rangle=v(R')$, a contradiction. Thus $t^b\in R'$ which is implies that $R[t^b]\subsetneq R'$. Since $v(R[t^b])=v(R')$. Then by Lemma~\ref{earinlefthand}, $R'=R[t^b]$. Because $0<b<a$,  the numbers of $b$ is finite and $|\Oc_R|<\infty$.

For the remaining cases, it is enough to show that   $|\Oc_R|<\infty$, if $v(R)=\langle 3,4\rangle$ or  $v(R)=\langle 3,5\rangle$.  Because if $v(R)=\langle 3,4,5\rangle$ or  $v(R)=\langle 3,5,7\rangle$, then they are overrings of rings  whose value semigroup is $\langle 3,4\rangle$ or  $\langle 3,5\rangle$, and so  $|\Oc_R|<\infty$,  as well.  We indicate the arguments for the case $H=\langle 3,4\rangle$. Let $R'\in \Oc_R$. We may assume that $R\subsetneq R'\subsetneq \overline{R}$. Then $v(R)\subsetneq v(R')$. Therefore, $v(R')=\langle 2,3\rangle$ or $v(R')=\langle3,4,5\rangle$. In the first case suppose that $t^2\not\in R'$. Then $R'\subsetneq R'[t^2]$, and so $v(R')\subsetneq v(R'[t^2])$. On the other hand, $\langle 2,3\rangle\subsetneq v(R'[t^2])$,  and since $1\not \in v(R'[t^2])$, it follows that $v(R')= v(R'[t^2])$, a contradiction. Thus $t^2\in R'$ and hence $R[t^2]\subseteq R'$. Since
$v(R[t^2])=v(R')$ we conclude that $R'=R[t^2]$.
The same argument shows that if  $v(R')=\langle3,4,5\rangle$, then $R'=R[t^5]$.
\end{proof}
We call $R$ {\em small} if $R$ has only finitely many trace ideals. In the next result the  small Gorenstein rings are classified.

\begin{Corollary}
\label{ProfessorKunz}
Assume that $|R/\mm|=\infty$, $R$ is  Gorenstein and $H=v(R)$. Then the following conditions are equivalent:
\begin{enumerate}
\item[(i)] $R$ is small.
\item[(ii)] $H=\langle 1\rangle$ or $H=\langle 2,a\rangle$ where $a>2$ is an odd number, or $H=\langle 3,4\rangle$ or $H=\langle 3,5\rangle$.
\item[(iii)] All trace ideals are integrally closed.
\end{enumerate}
\end{Corollary}

\begin{proof}
(i)\iff (ii): Since $R$ is Gorenstein, Proposition~\ref{gorensteinisdifficult} implies that  $|\Oc_R |<\infty$, if and only if $R$ is small. Hence the desired conclusion follows from  Theorem~\ref{ProfessorKunzlikestohelpsus}, if we observe that $v(R)$ cannot be $\langle 3,4,5\rangle$ or  $\langle 3,5,7\rangle$,  since $R$ is Gorenstein.

(i), (ii)\implies (iii): We have seen  in the proof of Theorem~\ref{ProfessorKunzlikestohelpsus}  that for  each gap of $H$ there exists exactly one proper overring of $R$. Therefore, Proposition~\ref{gorensteinisdifficult} implies that,  together with $R$, we have precisely $g(H)+1$ trace ideals. Since $H$ is symmetric, $n(H)=g(H)$, and therefore, the number of trace ideals is equal $n(H)+1$, which is the number of integrally  closed ideals.

(iii)\implies (i): Since the number of integrally closed  ideals is finite, (iii) implies that the number of trace ideals  is finite, that is, $R$ is small.
\end{proof}

\section{Binary operations on the set of trace ideals}
In this section we assume that $R$ is a one-dimensional analytically irreducible local domain.
By Proposition~\ref{whogetsvaccinefirst},  the sum of trace ideals is a trace ideal. However this is not the case for products of trace ideals. For example, $C$ is a trace ideal, but $C^2$ can not be a trace ideal if $R\neq \overline{R}$ because, in that case $C^2\subsetneq C$,  which by Proposition~\ref{conductor} is not possible if $C^2$ is a trace ideal.

 Therefore, we define a binary operation $*$  on $\Cc_R$ which replaces the product. In the case that $R$ is Gorenstein it has the property that the $*$-product  of trace ideals is again a trace ideal.

\medskip
Let $I, J\in \Cc_R$. We set
$
I*J=(I^{-1}J^{-1})^{-1}.
$
\begin{Proposition}
\label{star}
Let $I,J\in\Cc_R $. Then
\begin{enumerate}
\item[(a)] $I*J\in \Cc_R$.
\item[(b)] $I*J\subseteq (I^{-1})^{-1}\sect (J^{-1})^{-1}$.
\end{enumerate}
\end{Proposition}

\begin{proof}
(a)  Since $C\subseteq I, J\subseteq R$, it follows that $R\subseteq I^{-1}, J^{-1}\subseteq \overline{R}$. Therefore, $ R\subseteq I^{-1} J^{-1}\subseteq \overline{R}$, and hence $C\subseteq (I^{-1} J^{-1})^{-1}=I*J \subseteq R$.


(b) Since $I^{-1}J^{-1}\supseteq I^{-1}+J^{-1}$, it follows that $$I*J\subseteq (I^{-1}+J^{-1})^{-1}=(I^{-1})^{-1}\sect (J^{-1})^{-1}.$$
\end{proof}

In general,  the  intersection of trace ideals is not a trace ideal. But when $R$ is Gorenstein we will see that
 the $*$-product is the best   approximation of a trace ideal to the intersection. In the proof of the next results we use the fact that if $R$ is Gorenstein,  then the non-zero fractionary ideals are reflexive.

\begin{Proposition}
\label{bestapproximation}
Let $R$ be Gorenstein and let  $I,J\in \Tc_R $. Then
\begin{enumerate}
\item[(a)] $I*J\in \Tc_R$.
\item[(b)] $I*J=J*I$, $R*I=I*I=I$ and $I*C=C$.
\item[(c)] $I*J=\subtr(I\sect J)\subseteq I\sect J$.  In particular,  $I\sect J$ is a trace ideal if and only if $I\sect J=I*J$.
\end{enumerate}
\end{Proposition}

\begin{proof}
(a) By Proposition~\ref{gorensteinisdifficult},  $I^{-1}, J^{-1}\in\Oc_R$. Therefore, $I^{-1} J^{-1}\in\Oc_R$. Since $\alpha(I*J) =I^{-1}J^{-1}$,  Proposition~\ref{gorensteinisdifficult} implies that $I*J\in \Tc_R$.

(b) It is clear that $I*J=J*I$ and that $R*I=I*I=I$. By  Proposition~\ref{gorensteinisdifficult}, $I^{-1}$ is a ring. Therefore,  $I^{-1}I^{-1}=I^{-1}$. This shows that $I*I=I$. Finally,  $I*C=(I^{-1}C^{-1})^{-1}=(I^{-1}\overline{R})^{-1}=(\overline{R})^{-1}=C$.

(c) First note that  $I^{-1}+J^{-1}=(I^{-1}+J^{-1})^{-1})^{-1}= ((I^{-1})^{-1}\sect (J^{-1})^{-1})^{-1}=  (I\sect J)^{-1}$. Now let $L\in \Tc_R$ with $I*J\subseteq L\subseteq I\sect J$, then $I^{-1}+J^{-1}=( I\sect J)^{-1}\subseteq L^{-1}\subseteq I^{-1}J^{-1}$.
Since $L^{-1}$ is a ring, and since $I^{-1}\subseteq L^{-1}$ and $J^{-1}\subseteq L^{-1}$, it follows that $I^{-1}J^{-1}\subseteq L^{-1}$. Therefore, $I^{-1}J^{-1}=L^{-1}$ and hence $L=I*J$.
\end{proof}

The zero ideal $(0)$ is also a trace ideal. We set $(0)*I=I*(0)$.   The following result together with  Proposition~\ref{bestapproximation} implies  that if  $R=K[|H|]$  is a semigroup ring with $H$ a symmetric semigroup, then the set of monomial trace ideals of $R$  with the operations $*$ and $+$ has  the structure of a finite semiring.

\begin{Proposition}
\label{semiring}
Let $R$ be Gorenstein and let  $I,J, L\in \Tc_R $.
\begin{enumerate}
\item[(a)] $(I*J)*L=I*(J*L)$.
\item[(b)]  $J*I+L*I\subseteq (J+L)*I$. Equality holds, if $R$ is a semigroup ring and the ideals $I,J,L$ are monomial ideals in $R$.
\end{enumerate}
\end{Proposition}

\begin{proof}
(a) $(I*J)*L= (I^{-1} J^{-1})^{-1}*L=(((I^{-1} J^{-1})^{-1})^{-1}L^{-1})^{-1}=(I^{-1} J^{-1}L^{-1})^{-1}$. For the last equation we used that  $R$ is Gorenstein, which implies that $(A^{-1})^{-1}=A$ for any fractionary ideal $A$ of $R$. In the same we get $I*(J*L)=(I^{-1} J^{-1}L^{-1})^{-1}$.

(b) Let $I_1,I_2 \in\Cc_R $. Then $$I_1*I_2=(I_1^{-1}I_2^{-1})^{-1}= (I_1^{-1})^{-1}:I_2^{-1}=I_1:I_2^{-1}.$$
\[
(J+L)*I=(J+L):I^{-1}\supseteq J:I^{-1}+L: I^{-1}= I*J+ I*L.
\]
Equality holds, if $R$ is a semigroup ring and the ideals $I,J,L$ are monomial ideals in $R$.
\end{proof}

\medskip
\noindent
{\bf Acknowlegment}: The authors would like to thank Shinya Kumashiro for several useful comments, and especially for his comments regarding Proposition~\ref{gorensteindec22} and Theorem~\ref{ProfessorKunzlikestohelpsus}, and also want to thank the anonymous referee for his careful reading of the manuscript and pointing out some errors in the previous versions.


\begin{thebibliography}{99}


\bibitem{Be} R. Berger, Über eine Klasse unvergabelter lokaler Ringe,  Math. Annalen {\bf 146} (1962), 98--102.
\bibitem{BH} W. Bruns, J. Herzog, Cohen-Macaulay Rings, revised ed., Cambridge Stud. Adv. Math., vol. 39, Cambridge University Press, Cambridge, 1998.

\bibitem{Gap} M. Delgado, P. A. Garc\'{i}a-S\'{a}nchez  and J. Morais, numericalsgps: a GAP package on numerical
semigroups. (http://www.gap-system.org/Packages/numericalsgps.html).

\bibitem{GIK} S. Goto, R. Isobe and S. Kumashiro, Correspondence between trace ideals and birational
extensions with application to the analysis of the Gorenstein property of rings,
Journal of Pure and Applied Algebra  {\bf 224} (2020), 747--767.


\bibitem{HHS} J. Herzog, T. Hibi  and D. Stamate, Canonical trace ideal and residue for numerical semigroup rings, Semigroup Forum {\bf 103}(2021),  550–-566.
\bibitem{HHS1} J. Herzog, T. Hibi  and D. Stamate, The trace of the canonical module,  Israel Journal of Mathematics {\bf 233} (2019), 133--165.

\bibitem{HK} J. Herzog, E. Kunz,  Die Werthalbgruppe eines lokalen Rings der Dimension
$1$, Sitzungsberichte der Heidelberger Akademie der Wissenschaften (1971), Springer-Verlag

\bibitem{HK1} J. Herzog, E. Kunz, Der kanonische Modul eines Cohen--Macaulay Rings, LNM {\bf 238}, Springer, 1971


\bibitem{KT} T. Kobayashi, R. Takahashi, Rings whose ideals are isomorphic to trace ideals,  Math. Nachrichten {\bf 292}  (2019),  2252--2261.

\bibitem{Ku1} E. Kunz, The Value-Semigroup of a One-Dimensional Gorenstein Ring, Proc. Amer. Soc. {\bf 25} (1970), 748--751.

\bibitem{Ku} E. Kunz, Introduction to commutative algebra and algebraic geometry,  Springer Science \& Business Media, 1985.

\bibitem{Ku2} E. Kunz, Introduction to plane algebraic curves. {\bf 68} (Translated by R.G.Belshoff)  Boston: Birkhäuser, 2005.

\bibitem{Li} L.  Lindo,   Trace ideals and centers of endomorphism rings of modules over commutative rings, J. Algebra {\bf 482} (2017), 102--130.

\bibitem{LP} L. Lindo, N. Pande, Trace ideals and the Gorenstein property, arXiv:1802.06491v2.

\bibitem{Mats} H. Matsumura, Commutative Ring Theory, Cambridge University Press, 1971.


\end{thebibliography}
\end{document}